\documentclass[11pt,reqno]{amsart}

\usepackage{latexsym}
\usepackage{amsfonts}
\usepackage{amsmath}
\usepackage{url}
\usepackage{amssymb,amsmath}

\def\bl{\begin{lemma}}
\def\el{\end{lemma}}
\def\bth{\begin{theorem}}
\def\eth{\end{theorem}}
\def\bc{\begin{corollary}}
\def\ec{\end{corollary}}
\def\bcj{\begin{conjecture}}
\def\ecj{\end{conjecture}}
\def\bpr{\begin{proposition}}
\def\epr{\end{proposition}}
\def\bde{\begin{definition}}
\def\ede{\end{definition}}
\def\E{\mathbb{E}}

\newcommand{\dist}{\mbox{\rm dist}}
\newcommand{\be}{\begin{eqnarray}}
\newcommand{\ee}{\end{eqnarray}}
\newcommand{\eps}{{\mbox{$\epsilon$}}}

\newcommand{\Z}{{\mathbb Z}}

\newcommand{\F}{{\mathcal F}}

\newcommand{\VV}{{\mathcal V}}

\newcommand{\EE}{{\mathcal E}}

\renewcommand{\and}{\hbox{ {\rm and} }}

\newcommand{\C}{{\mathcal{C}}}
\newcommand{\prob}{\mbox{\bf P}}

\newcommand{\lr}{\leftrightarrow}

\newcommand{\La}{{\mathcal{L}}}
\newcommand{\wg}{{\widetilde{g}}}

\newtheorem{theorem}{Theorem}[section]
\newtheorem{definition}{Definition}[section]
\newtheorem{lemma}[theorem]{Lemma}

\newtheorem{corollary}[theorem]{Corollary}
\newtheorem{proposition}[theorem]{Proposition}
\newtheorem{conjecture}[theorem]{Conjecture}

\theoremstyle{definition}
\numberwithin{equation}{section}

\begin{document}
\title{Is the critical percolation probability local?}
\author{Itai Benjamini, Asaf Nachmias \and Yuval Peres}

\begin{abstract}
We show that the critical probability for percolation on a
$d$-regular non-amenable graph of large girth is close to the
critical probability for percolation on an infinite $d$-regular
tree. This is a special case of a conjecture due to O. Schramm on
the locality of $p_c$. We also prove a finite analogue of the
conjecture for expander graphs.
\end{abstract}

\maketitle

\section{Introduction}
Denote by $p_c(G)$ the critical probability for Bernoulli bond
percolation on an infinite graph $G$, that is,
$$p_c(G) = \inf \Big \{ p \in [0,1] : \prob_p \big ( \exists \, \, \hbox{an infinite component} \big ) >0 \Big \} \, .$$
Is the value of $p_c$ determined by the local geometry of the
graph or by global properties (such as volume growth and
expansion)? In this note we show that the former is the correct
answer for non-amenable graphs with tree-like local geometry, and
discuss a conjecture of Schramm that $p_c$ is locally determined
in greater generality.

Recall that the {\em girth} $g$ of a graph $G$ is the minimum
length of a cycle in $G$. Let $P$ be the transition matrix of the
simple random walk (SRW) on $G$ and let $I$ be the identity matrix. The
{\em bottom of the spectrum} of $I-P$ is defined to be the largest
constant $\lambda_1$ with the property that for all $f \in \ell^2(G)$ we have
\be \label{radius} \langle f , (I-P)f \rangle
\geq \lambda_1 \langle f,f \rangle \, .
\ee
Kesten (\cite{Kes1}, \cite{Kes2}) proved that $G$ is a {\em non-amenable} Cayley graph if and only if $\lambda_1>0$. 
This was extended by Dodziuk \cite{D} to general infinite bounded degree graphs (for more background on non-amenability see \cite{LP} and \cite{W}).

\begin{theorem} \label{localpc} There exists an absolute constant $C>0$ such that if $G$
is a non-amenable regular graph with degree $d$ and girth $g$ such
that the bottom of spectrum of $I-P$ is $\lambda_1>0$, then
$$ p_c(G) \leq {1 \over d-1} + {C \log\big( 1 + {1 \over  \lambda_1^2} \big) \over dg} \, .$$
\end{theorem}
Recall that $p_c(T_d)={1 \over d-1}$ where $T_d$ is an infinite
$d$-regular tree and that for any $d$-regular graph $G$ we have $p_c(G)\geq {1 \over d-1}$. Thus, Theorem \ref{localpc} asserts that
non-amenable graphs with large girth and degree $d$ have $p_c$ close to the lowest possible value, $p_c(T_d)$.

It is easy to construct non-amenable graphs with arbitrary girth,
for example, take the Cayley graph of $\langle a,b,c \mid
c^n=1\rangle$. Olshanskii and Sapir \cite{OS} constructed for any
$k\geq 2$ a group $\Gamma_k$ with the following property. For any
$\ell>0$ there is a set $S_\ell$ consisting of $k$ generators
for $\Gamma_k$, such that the Cayley graph $G(\Gamma_k, S_\ell)$
has girth at least $\ell$ and $\inf_\ell \lambda_1( G(\Gamma_k,
S_\ell))>0$ (as remarked in \cite{OS}, for $k \geq 4$ such groups were  also constructed by Akhmedov \cite{AK}).

Let $G$ be a graph and $v$ a vertex in $G$. Denote by $B_G(v,R)$ the ball of radius $R$ in $G$
centered at $v$, in the graph metric, with its induced graph structure.
%If $G$ is transitive, we denote the isomorphism class of $B_G(v,R)$
%by $B_G(R)$, since it does not depend on $v$.
We say that a sequence of transitive graphs $G_n$ {\em converges} to $G$ if for any integer $R>0$ there exists $N$
such that $B_{G_n}(v_n, R)$ and $B_G(v,R)$ are isomorphic as rooted graphs, for all $n\geq N$ (note that the choices of $v_n$ and $v$ are irrelevant due to transitivity). Oded Schramm (personal communication) suggested the
following conjecture.

\begin{conjecture}\label{reallocal}
Let $G_n$ be sequence of vertex transitive infinite graphs with $\sup_n p_c(G_n)<1$
such that $G_n$ converges to a graph $G$.
Then $p_c(G_n) \rightarrow p_c(G)$.
\end{conjecture}
This conjecture is open for infinite graphs even if we assume that they are uniformly nonamenable.
We can prove the following analogue of the
conjecture for finite expander graphs, by extending the analysis
of  \cite{ABS}, Proposition $3.1$. For each $n \ge 1$, let $G_n$ be a {\em
finite} graph and let $U_n$ be a uniformly chosen random vertex in $G_n$. We say
that the sequence of {finite} graphs $\{G_n\}$ {\em converges weakly}
to an infinite rooted graph $(G,\rho)$ (where $\rho$ is a fixed
vertex of $G$) if for each $R>0$ we have
$$ \prob \Big ( B_{G_n}(U_n, R) \neq B_{G}(\rho, R) \Big ) \to 0 \, \qquad \hbox{as } n \to \infty \, ,$$
where the event above means that the balls are not isomorphic as rooted graphs.
%where  $B_G(\rho, R)$ are the  balls centered at  $U_n$ and $\rho$ in $G_n$ and $G$, respectively.
This is a special case of the graph limits defined in
\cite{BS2}. For two sets of vertices $A$ and $B$, write $E(A,B)$
for the set of edges with one endpoint in $A$ and the other in
$B$. Recall that the Cheeger constant $h(G)$ of a finite graph
$G=(V,E)$ is defined by
$$ h(G) = \min _{A \subset V} \Big \{ {|E(A,V\setminus A)| \over |A|} \, :\,  0 < |A| \leq |V|/2 \Big \} \, .$$

\begin{theorem} \label{abs} Let $(G,\rho)$ be an infinite bounded degree rooted graph and let $G_n$ be a sequence
of finite graphs with uniform Cheeger constant $h>0$ and a uniform
degree bound $d$, such that $G_n \to G$ weakly. Let $p \in [0,1]$
and write $G_n(p)$ for the graph of open edges obtained from $G_n$
by performing bond percolation with parameter $p$. If $p<p_c(G)$,
then for any constant $\alpha>0$ we have
$$ \prob \Big ( G_n(p) \hbox{ contains a component of size at least } \alpha|G_n|  \Big ) \to 0 \quad \hbox{as } n\to \infty \, ,$$
and if $p>p_c(G)$, then there exists some $\alpha>0$ such that
$$ \prob \Big ( G_n(p) \hbox{ contains a component of size at least } \alpha|G_n|  \Big ) \to 1 \quad \hbox{as } n\to \infty \, .$$
\end{theorem}

For the reader's convenience, we present the proof of Theorem
\ref{abs} in Section \ref{abspf}. 
%We remark that the proof does
%not seem to adapt to the infinite case, and currently we do not
%know if Conjecture \ref{reallocal} holds even under the stronger
%assumption of non-amenability.

\subsection{Further discussion.} Conjecture \ref{reallocal} suggests that the critical percolation probability
is locally determined. This contrasts with critical exponents
which are believed to be {\em universal} and depend only on global
properties of the graph. For instance, the value of $p_c$ on the
the two dimensional square lattice is ${1 \over 2}$, but on the two
dimensional triangular lattice it is $2{\sin(\pi/18)}$; however,
the critical exponents are believed to be the same.

It is worth noting another example of the locality of $p_c$ where the limit graph
is the lattice $\Z^d$. For
$d>1$ and $n>1$, write $\Z_n^d$ for the $d$-dimensional torus with
side $n$. The following theorem is an immediate corollary of a theorem of Grimmett and Marstrand \cite{GM} combined with
the fact that the critical probability of a quotient graph is always at least the critical probability of the original graph
(see \cite{BS}, \cite{C} or \cite{LP}).
\begin{theorem} [Grimmett, Marstrand \cite{GM}] For any $d>1$ and $k$ satisfying $2 \leq k < d$ we have
$$ p_c ( \Z^k \times \Z_n^{d-k} ) \to p_c(\Z^d) \quad \hbox{as } n \to \infty \, .$$
\end{theorem}
Observe that this is theorem is a special case of Conjecture \ref{reallocal}. For background and further conjectures regarding
percolation on infinite graphs see \cite{BS, LP}.

%For instance,
%the value of $p_c$

%In this note we show that if a graph $G$ of degree $d$ locally
%looks like a tree (that is, has high girth), then $p_c$ is close
%to ${1 \over d-1}$, i.e., the value of $p_c$ on a $d$-regular
%tree. We prove this with the additional assumption of
%non-amenability, which we believe is not necessary, see Conjecture
%\ref{reallocal}.

%\note{We should ask Mark what examples he exactly has}.

%\section{Proof of Theorem \ref{localpc}} \label{localpcpf}
\section{Uniform escape probability}
In this section we prove a useful lemma.
\begin{lemma} \label{return} Consider a reversible irreducible Markov chain $\{X_t\}$ on a countable state space $V$, with infinite stationary measure $\pi$  and transition matrix $P$, such that the bottom of the spectrum of
$I-P$ is $\lambda_1>0$ (that is, (\ref{radius}) holds for any $f\in
\ell^2(\pi)$).
%The latter assumption means that for all $f \in \ell^2(\pi)$ we have
%\be \label{radius} \langle f , (I-P)f \rangle \geq \lambda_1 \langle f,f \rangle \, ,\ee
%and $\lambda_1$ is the largest constant with this property.
 Let $A \subset V$ be a nonempty set of states with $\pi(A)<\infty$ and let $\pi_A(\cdot) =\pi(\cdot)/\pi(A)$ be the normalized restriction of $\pi$ to $A$. Then
$$ \prob_{\pi_A} \Big ( X_t \hbox{ never returns to } A \Big ) \geq \lambda_1 \, .$$
\end{lemma}
\begin{proof}
%We will prove the statement for the simple random walk, and it follows for the non-backtracking random walk since once can couple (by erasing the %SRW's backtracks) the two walks such that the range of the non-backtracking walk is contained in the range of the simple random walk.
%Fix a large integer $K>0$ and put
%$$ B_K = \{ v \in G : d(v, A) \geq K \} \, , $$
%where $d(v,A) = \min _{a \in A} d(v,a)$ and $d(v,a)$ is the minimal $t$ such that $p^t(v,a)>0$. We also put
Let $B \subset V$ be disjoint from $A$ such that  $V\setminus (A \cup B)$ is finite. Define
$$ \tau = \min\{ t \geq 0 : X_t \in A \cup B \} \, \, \qquad \hbox{and} \qquad \tau^+ = \min\{ t > 0 : X_t \in A \cup B \} \, .$$
The irreducibility assumption and the finiteness of the complement
of $A \cup B$ imply that $\tau^+ <\infty$ a.s.\ for any starting
state. We will show that for all sets $B$ as above, \be
\label{eqb} \prob_{\pi_A} \big ( X_{\tau^+} \in B \big ) \geq
\lambda_1 \, .\ee The assertion of the lemma then follows by
enumerating $V\setminus A$ as $v_1,v_2, v_3, \ldots$, taking
$B=B_k=\{v_j \, : \,  j \ge k\}$ and intersecting the events in
(\ref{eqb}) over all these sets $B=B_k$ for $k \ge 1$. Let
$$ f(x) = \prob _x ( X_\tau \in A ) \, .$$
Observe that $f \equiv 1$ on $A$ and $f \equiv 0$ on $B$.  For all $x\in G$,
$$ (Pf)(x) = \prob_x \big ( X_{\tau^+} \in A \big ) \, .$$
In particular, $f$  is harmonic (satisfies $Pf=f$)
 on $G \setminus (A\cup B)$. Thus
$((I-P)f)(x) = \prob _x \big ( X_{\tau^+} \in B \big )$ for $x\in A$ and $((I-P)f)(x)=0$ for $x\in G\setminus(A \cup B)$. Therefore
$$ \langle f , (I-P)f \rangle = \sum_{x \in A} \pi(x) \, \prob _x \big ( X_{\tau^+} \in B \big ) = \pi(A)\, \prob_{\pi_A} \big ( X_{\tau^+} \in B \big ) \, .$$
On the other hand, clearly,
$$ \langle f,f \rangle \geq \sum_{x\in A} \pi(x) f(x)^2 = \pi(A) \, .$$
The claim (\ref{eqb}) follows by inserting the last two formulas
in (\ref{radius}).
\end{proof}

\section{Proof of Theorem \ref{localpc}} \label{localpcpf}
We return to the setting of Theorem \ref{localpc}.
%Recall that for a finite vertex set $A \subset G$ we write
%$\partial A$ for the vertices not in $A$ which have a neighbor in $A$.
%We write $\bar{A}$ for $A \cup \partial A$ and for $\alpha>0$,
Let $G$ be regular graph of degree $d$ and girth $g$ and write $\wg:=\lceil g/2 \rceil-1$. Given a set of vertices $A$ in $G$ and
 $\alpha \in (0,1)$, we say that an edge $(x,u)$ is $(\alpha,A)$-{\bf
good} if $x\in A$ and at least an $\alpha$ fraction of the $(d-1)^{\wg}$ non-backtracking
paths of length $\wg$ emanating from $u$, for which the first step is not $x$, avoid $A$ (in particular,
$u \not \in A$.) The following lemma is a corollary of Lemma \ref{return}.

\begin{corollary} \label{goodvertices} Let $G$ be a regular graph with degree $d$ and girth
$g$. If $G$ is nonamenable, i.e., it satisfies $\lambda_1>0$, then for any finite set $A \subset V(G)$,
there exist at least ${\lambda_1 d \over 2} |A|$ edges $(x,u)$
which are $(\lambda_1/2,A)$-good.
\end{corollary}
\begin{proof} For an edge $(x,u)$ with $x \in A$, let
$\beta_{(x,u)}= \prob_x(X_1=u \and \forall t>0 \; \, X_t \not \in A )$,
where $\{X_t\}$ is a SRW in $G$, started at $x$.
%By conditioning on the first step of the walk we get that
%$$ \beta_x = {1 \over d} \sum_{u : (x,u)\in E(G)} \prob_u (\forall t\geq0 \; \, X_t \not \in A ) \, .$$
%Hence, if $\beta_x \geq \alpha$, then there are at least $\alpha d /2$ vertices $u$ such that $u$ is a neighbor of $x$ and $\prob_u (\forall t\geq0 \; \, X_t \not \in A ) \geq \alpha/2$.
Let $\tau:=\min\{t \,:\, \dist(u,X_t)=\wg \}$.
Since the ball $B_G(u,\wg)$ is a spherically symmetric tree, the loop erasure of $(X_t)_{t=0}^\tau$ yields a uniform random non-backtracking path of length $\wg$ from $u$. Thus if $\beta_{(x,u)} \ge \alpha$, then the edge $(x,u)$ is $(\alpha, A)$-good.
By Lemma \ref{return},
$$ {1 \over d|A|} \sum_{(x,u) : x \in A} \beta_{(x,u)} \geq \lambda_1 \, ,$$
and we conclude that at least ${\lambda_1 d \over 2}|A|$ edges $(x,u)$ with
$x \in A$ must satisfy $\beta_{(x,u)} \geq \lambda_1/2$.
\end{proof}

\noindent {\bf Proof of Theorem \ref{localpc}.} Let $\eps>0$ be a
small number and set $p={1 \over d-1}+\eps$. For each edge $e$ we
draw two independent Bernoulli random variables $X_e(p)$ and
$Y_e(\eps)$ with means $p$ and $\eps$ respectively. We say that an
edge is {\em open} if one of these variables takes the value $1$
and {\em closed} otherwise. We also say that the edge $e$ is
$p$-open if $X_e(p)=1$ and $\eps$-open if $Y_e(\eps)=1$. For a
vertex $v$ we write $\C(v)$ for the open cluster of $v$. The
probability that an edge is closed is $(1-p)(1-\eps)$, hence
$|\C(v)|$ is dominated by the cluster size in $(p+\eps)$-bond
percolation. Our goal is to show that with positive probability
$|\C(v)|=\infty$.

We perform the following exploration process, which will produce
an increasing sequence $\{A_t\}$ of connected vertex sets in which
$A_t \subset \C(v)$ for all $t$. At each step, some of the edges
touching $A_t$ will be $\eps$-closed and some will be
$\eps$-unchecked. We begin by setting $A_0$ to be the $p$-cluster
of $v$ (that is, all the vertices connected to $v$ by $p$-open
paths) and all the edges touching $A_0$ are $\eps$-unchecked. We
assume that $A_0$ is finite (otherwise we are finished). At step
$t>1$ let $\EE_{t-1}$ be the set of $\eps$-unchecked edges $(x,u)$ such
that $(x,u)$ is $(\lambda_1/2,A_{t-1})$-good. If
$\EE_{t-1}$ is empty, the process ends. If not, we choose $(x,u)\in
\EE_{t-1}$ according to some prescribed ordering of the edges
and check whether the edge is $\eps$-open. If it is $\eps$-closed we put
$A_t = A_{t-1}$ and continue to the next step of the process.
Otherwise,  we let
$$A_t = A_{t-1} \cup \VV \, ,$$
where $\VV$ is the set of vertices $v$ of distance at most
$\wg$ from $u$ such that the unique
path of length at most $\wg$ between $u$ and $v$ avoids $A_{t-1}$ and is $p$-open.

%
% consider
%the vertices $\{v_1, \ldots, v_k\}$ of distance at most $\wg$ from
%$x$ for which the unique path of length at most $\wg$ between $x$
%and $v_i$ goes through $u$ and avoids $A_{t-1}$. We check for
%which $i$'s the path between $x$ and $v_i$ is $p$-open and for
%such $i$'s we add $v_i$ to $A_t$.
This finishes the description of the exploration process. To
analyze this process we introduce the following random variable
%$$ Y_t = \big | \big \{ e : e \hbox{ is an $\eps$-unchecked boundary edge of $A_t$} \big \} \big | \, ,$$
$$ Z_t = \big | \big \{ e : e \hbox{ is an $\eps$-closed and $\eps$-checked edge touching $A_t$} \big \} \big | \, .$$
%At each step, if a boundary edge turns out to be $\eps$-open we add it to the set, hence there are no $\eps$-open boundary edges and we have
%$$ Y_t + X_t = |\partial A_t| \, .$$
Let $\tau$ be the stopping time
$$ \tau = \min \Big \{ t : |A_t| < {2t \over \lambda_1 d} \Big \} \, .$$
At each step we check the $\eps$-status precisely one edge, hence
$Z_t \leq t$ for all $t$. Thus, by Corollary \ref{goodvertices}, if
$|A_t| > {2t \over \lambda_1 d}$ there must exist at least one
$\eps$-unchecked edge $(x,u)$ which is $(\lambda_1/2, A_t)$-good.
Write $\F_t$ for the $\sigma$-algebra generated by the $\eps$ and
$p$ status of the edges we examined in the exploration process up
to time $t$ and let $\xi_t = |A_{t+1}| - |A_t|$. By the discussion
above we have that
\begin{eqnarray} \label{lowerexp} \E [ \xi_t \mid\F_{t-1}, \tau > t ] &\geq& {\eps \lambda_1 \over 2} \sum_{j=1}^\wg (1+\eps(d-1))^j
\geq {\lambda_1 [ (1+\eps(d-1))^\wg - 1 ] \over 2(d-1)} \, .
\end{eqnarray}
%\sum_{j=1}^{{\wg}} (1+\delta)^j \geq
% {\eps \delta^{-1} \lambda_1 [(1+\delta)^{\wg} -1]\over 4} \, ,
To see the first inequality in (\ref{lowerexp}), recall that
$(x,u)$ is $\eps$-open with probability $\eps$. Also, for any $j
\leq \wg$ the expected number of vertices of distance $j$ from $u$
such that the path between them and $u$ avoids $A_t$ and is
$p$-open is at least ${\lambda_1 \over 2} (p(d-1))^j={\lambda_1
(1+\eps(d-1))^j \over 2}$.
%Given such a vertex, the probability that it is added to $\partial
%A_{t+1}$ is $(1-p^{d-1})$, that is, the probability that one of
%its $d-1$ outward edges is $p$-closed.

%Note that we have
%can take $\{\xi_t\}_{t \geq 0}$ to be a
%martingale difference sequence, that is
%$$ \E \Big [ \xi_t - \E\xi_t \, \mid \, \xi_{t-1},\ldots, \xi_1 \Big ] = 0 \, ,$$
%and that $|\xi_t| \leq (d-1)^{\wg}$.
We now assume that \be\label{girthbound}\wg \geq {\log\big (
1+{8\over \lambda_1^{2}} \big ) \over \log(1+\eps(d-1)) } \, ,\ee
so that $\E [\xi_t \mid \F_{t-1}, \tau > t] \geq 4
d^{-1}\lambda_1^{-1}$ by (\ref{lowerexp}). Since $|\xi_t| \leq
(d-1)^{\wg}$, Azuma-Hoeffding's inequality (see Chapter $7$ of
\cite{AS}) gives that for any $t>1$ \be\label{ld} \prob \big (
\tau = t+1 \, \mid \, A_0 \big ) \leq \prob \Big ( \sum_{i=1}^t
\xi_i \leq {2t \over d\lambda_1} \Big ) \leq e^{-ct} \, ,\ee where
$c=2\lambda_1^{-2}d^{-2}(d-1)^{-2\wg}>0$. Since $|A_t|$ is a
non-decreasing sequence we have that $\tau > {\lambda_1 d |A_0|
\over 2}$.
%at each step
%$|\partial A_t|$ can decrease by at most $1$, we have that
%$$ |\partial A_t| \geq |\partial A_{t-1}| -1 \geq  |\partial A_0| - t \, ,$$
%whence $\tau \geq |\partial A_0|/(4\lambda_1^{-1} +1)$.
For any $K>0$ there is some positive probability (depending on
$K$) of having $|A_0| \geq K$ and we infer from (\ref{ld}) that
\begin{eqnarray*} \prob (\tau = \infty) &\geq& \prob (|A_0| \geq K) \prob \big (\tau = \infty \, \mid \, |A_0| \geq K \big ) \\
&\geq& \prob (|A_0| \geq K) \Big [ 1- \sum_{t \geq {\lambda_1 d K
\over 2}} e^{-ct} \Big ] > 0 \, ,
\end{eqnarray*}
as long as we choose $K=K(g, \lambda_1, d)$ to be large enough.
The event $\tau = \infty$ implies that $|\C(v)|=\infty$, and
hence, by (\ref{girthbound}) when $\eps \geq C (dg)^{-1} \log\big(
1 + {1 \over \lambda_1^2} \big)$ there is positive probability of
an infinite component in $\big ( {1 \over d-1} + 2\eps \big
)$-bond percolation (for $\eps \leq (d-1)^{-1}$ one can take $C=128$ using the inequalities $\log(1+8x)\leq 8\log(1+x)$ and $\log(1+x)\geq x/2$ valid for $x\in(0,1)$). This concludes the proof of the theorem.
 \qed

\section{Proof of Theorem \ref{abs}}\label{abspf}
Without loss of generality assume that $|G_n|=n$. We first take $p < p_c(G)$ and fix $\alpha>0$. Since $p<p_c(G)$, for any $\eps>0$ there exists $R=R(\eps)$ large enough such that in $G$
$$ \prob_p \big ( \rho \lr \partial B(\rho, R) \big ) < \eps \, .$$
Thus for large enough $n$ we have in $G_n$
$$ \La \times \prob_p \Big ( U_n \lr \partial B(U_n, R) \Big ) \leq \eps \, ,$$
where $\La$ is the law of $U_n$. Since $G$ has bounded degree, we
deduce that for any $\eps>0$ there exists $n$ large enough
such that
$$ \La \times \prob_p \Big ( |\C(U_n)| \geq d^{R+1}  \Big ) \leq \eps \, .$$
Write $\C_1(n)$ for the largest component of $G_n(p)$ and note
that as long as $d^{R+1} \leq  \alpha n$ we have that
$$ \La \times \prob_p \Big ( |\C(U_n)| \geq d^{R+1}  \Big ) \geq \alpha \prob_p \big (  |\C_1(n)| \geq \alpha n \big ) \, ,$$
and we get that
$$ \prob_p \big (  |\C_1(n)| \geq \alpha n \big ) \leq {\eps \alpha^{-1} }\, ,$$
which proves the first assertion of the theorem. \\

To prove the second assertion of the theorem we use a sprinkling
argument, as in \cite{ABS}. Assume $p > p_c(G)$ and for some
$\eps>0$ let $p_1 = p_c(G) + \eps$ such that
$1-p=(1-p_1)(1-\eps)$. We first consider $G_n(p_1)$. Since $p_1 >
p_c(G)$, there exists some $\delta>0$ such that for all $R>0$ we
have
$$ \prob _{p_1} \big ( \rho \lr \partial B(\rho, R) \big ) \geq \delta \, .$$
For $v \in G_n$, write $B_{p_1}(v,R)$ for the set of vertices in
$G_n(p_1)$ which are connected to $v$ in a $p_1$-open path of
length at most $R$. We get that for any $R>0$ there exists $n_0$
such that for $n\geq n_0$ we have in $G$
$$ \La \times \prob\Big ( |B_{p_1}(U_n, R)| \geq R \Big ) \geq \delta/2 \, .$$
Let $X_R$ denote the random variable
$$ X_R = \Big | \Big \{ v \in G_n : |B_{p_1}(v,R)| \geq R \Big \} \Big | \, ,$$
so that $\E X_R \geq \delta n/2$. On the other hand, note that
changing the status of a single edge can change $X_R$ by at most
$d^R$, where $d$ is the degree bound of $G_n$. The method of
bounded differences (see Theorem 3.1 of \cite{Mc} or Chapter 7 of
\cite{AS}) gives that \be\label{diff} \prob \Big ( X_R \leq
{\delta n \over 4} \Big ) \leq \exp\big ( -{\delta^2 n \over 8
d^{2R}} \big ) \, .\ee

Assume now that $X_R \geq \delta n/4$ and consider the connected
components of $G_n(p_1)$ of size at least $R$. Their number $m$ is
at most $n/R$. We now consider the union of $G_n(p_1)$ with
$G_n(\eps)$ and claim that many of these components join together
by edges of $G_n(\eps)$ and create a component of linear size.
Indeed, consider a partitition of the $m$ large components of
$G_n(p_1)$ into two sets, $A$ and $B$, each spanning at least
$\delta n / 12$ vertices. If for any such partition there is an
open path in $G_n(\eps)$ connecting $A$ and $B$, then there exists
a component of size at least $\delta n /12$ in $G_n(p_1) \cup
G_n(\eps)$. Since $G_n$ has Cheeger constant at least $h>0$ we get
by Menger's Theorem that for any such $A$ and $B$ there are at
least $h\delta n /12$ edge disjoint paths connecting $A$ to $B$.
Since there are most $d n/2$ edges in $G_n$ we have that at least
a half of these paths must be of length at most ${12d \over h
\delta}$. The probability that all these paths are closed in
$G_n(\eps)$ is at most
$$ \Big [ 1 - \eps^{12d \over h\delta} \Big ] ^{h\delta n \over 24} \, .$$
There are at most $2^m$ different possibilities for choosing $A$ and $B$.
Hence, the probability that there exists such $A$ and $B$ is at most
$$ 2^m \Big [ 1 - \eps^{12d \over h\delta} \Big ] ^{h\delta n \over 24} \leq \exp \Big ( {n/R - \eps^{12d \over h \delta} h \delta n/24} \Big)  \, ,$$
which goes to $0$ as long as $R$ is chosen such that $R^{-1} < \eps^{12 d \over h\delta}h\delta$. This together with (\ref{diff}) shows that
$$ \prob \Big ( G_n(p) \hbox{ contains a component of size at least } {\delta n \over 12} \Big ) \to 1 \qquad \hbox{as } n\to \infty \, .$$

 \qed \\

\noindent {\bf Acknowledgements:} We are indebted to Mark Sapir and Oded Schramm for useful
discussions.

\vspace{.05 in}\noindent
{\bf Itai Benjamini}: \texttt{itai.benjamini(at)weizmann.ac.il} \\
The Weizmann Institute of Science, \\
Rehovot POB 76100, Israel.

\medskip \noindent
{\bf Asaf Nachmias}: \texttt{asafn(at)microsoft.com} \\
Microsoft Research,
One Microsoft way,\\
Redmond, WA 98052-6399, USA.

\medskip \noindent
{\bf Yuval Peres}: \texttt{peres(at)microsoft.com} \\
Microsoft Research,
One Microsoft way,\\
Redmond, WA 98052-6399, USA.


\begin{thebibliography}{BKC}

\bibitem{ABS}
N. Alon, I. Benjamini and A. Stacey, Percolation on finite graphs and isoperimetric inequalities {\em  Ann. Probab.} {\bf 32} (2004), 1727--1745.

\bibitem{AK} A. Akhmedov, The girth of groups satisfying Tits Alternative, {\em J. of Algebra}, {\bf 287} (2005), no.2, 275-282.

\bibitem{AS} N. Alon and J. H. Spencer, The probabilistic method,
2nd edition, Wiley, New York, 2000.

\bibitem{BS} I. Benjamini and O. Schramm, Percolation Beyond $Z^d$, Many Questions And a Few Answers, {\em ECP}, {\bf 1} (1996), Paper no. 8.

\bibitem{BS2} I. Benjamini and O. Schramm, Recurrence of Distributional Limits of Finite Planar Graphs, {\em Electron. J. Probab.} Vol. 6, no. 23, 13 pp.

\bibitem{C} M. Campanino, Inequalities for critical probabilities in percolation, {\em Particle
Systems, Random Media and Large Deviations}, volume 41 of Contemp.
Math., pages 1-9, R. Durrett editor. Amer. Math. Soc., Providence, RI. Proceedings of the AMSIMS-
SIAM joint summer research conference in the mathematical sciences on
mathematics of phase transitions held at Bowdoin College, Brunswick, Maine,
June 24-30, (1984).

\bibitem{D} J. Dodziuk, Difference equations, isoperimetric inequality and transience of certain random walks, {\em Trans. Amer. Math. Soc.},
{\bf 284}, 787-794. 

\bibitem{GM} G. R. Grimmett and J. M. Marstrand, The supercritical phase of percolation is well behaved, {\em Proc. Roy. Soc. London} Ser. A 430 (1990), no. 1879, 439--457.

\bibitem{Kes1} H. Kesten, Full Banach mean values on countable groups, {\em Math. Scand.}, {\bf 7}, 146-156.

\bibitem{Kes2} H. Kesten, Symmetric random walks on groups, {\em Trans. Amer. Math. Soc.}, {\bf 92}, 336-354.

\bibitem{LP}
R. Lyons with Y. Peres, Probability on Trees and Networks,
In preparation, \url{http://mypage.iu.edu/~rdlyons/prbtree/prbtree.html}

\bibitem{Mc} C. McDiarmid, Concentration. Probabilistic methods for algorithmic discrete mathematics, 195--248, Algorithms Combin., 16, Springer, Berlin, (1998).

\bibitem{OS} A. Yu. Olshanskii and M. V. Sapir, On $k$-free-like groups, preprint. Available at \url{http://arxiv.org/abs/0811.1607}

\bibitem{W} W. Woess, Random walks on infinite graphs and groups, {\em Cambridge Tracts in Mathematics}, {\bf 138}, Cambridge University Press, Cambridge.

\end{thebibliography}
\end{document}